\documentclass[12pt, a4paper, abstracton, bibliography=totoc]{scrartcl}
\pdfoutput=1

\usepackage{amsmath, amsthm, amsfonts, amssymb}
\usepackage[utf8]{inputenc}
\usepackage[T1]{fontenc}
\usepackage{color}
\usepackage{slashed} 
\usepackage[all, cmtip]{xy}
\usepackage{hyperref} 

\let\counterwithout\relax

\usepackage{chngcntr} 

\usepackage{graphicx} 
\usepackage[english]{babel}
\usepackage{microtype}
\usepackage{bm} 
\usepackage{xfrac} 
\usepackage{mathtools} 

\newcommand{\IN}{\mathbb{N}}
\newcommand{\IZ}{\mathbb{Z}}

\newcommand{\IR}{\mathbb{R}}
\newcommand{\IC}{\mathbb{C}}

\newcommand{\diam}{\operatorname{diam}}

\newcommand{\id}{\mathrm{id}}

\newcommand{\HCocont}{H \! C^\mathrm{cont}}
\newcommand{\HHocont}{H \! H^\mathrm{cont}}

\newcommand{\HC}{H \! C}

\newcommand{\HH}{H \! H}

\newcommand{\Frechet}{Fr\'{e}chet }

\newcommand{\RD}{\mathrm{RD}}

\newcommand{\RNum}[1]{\uppercase\expandafter{\romannumeral #1\relax}}

\DeclareMathOperator{\hatotimes}{\hat{\otimes}}
\DeclareMathOperator{\barotimes}{\bar{\otimes}}

\newtheorem{thm}{Theorem}[section]
\newtheorem*{thm*}{Theorem}
\newtheorem*{mainthm*}{Main Theorem}
\newtheorem{cor}[thm]{Corollary}
\newtheorem*{cor*}{Corollary}
\newtheorem{lem}[thm]{Lemma}
\newtheorem{prop}[thm]{Proposition}

\theoremstyle{definition}
\newtheorem{rem-alt}[thm]{Remark}
\newtheorem*{rem*}{Remark}
\newtheorem{ex-alt}[thm]{Example}

\newtheorem*{example*}{Example}
\newtheorem*{examples*}{Examples}
\newtheorem{defn-alt}[thm]{Definition}
\newtheorem{nota-alt}[thm]{Notation}

\newenvironment{defn}    
{%
	\pushQED{\qed}\begin{defn-alt}}
	{\popQED\end{defn-alt}}
	
\newenvironment{ex}    
{%
	\pushQED{\qed}\begin{ex-alt}}
	{\popQED\end{ex-alt}}

\newenvironment{rem}    
{%
	\pushQED{\qed}\begin{rem-alt}}
	{\popQED\end{rem-alt}}

{%
	\pushQED{\qed}\begin{nota-alt}}
	{\popQED\end{nota-alt}}

\numberwithin{equation}{section}
\counterwithout{footnote}{section}

\makeatletter
\def\blfootnote{\gdef\@thefnmark{}\@footnotetext}
\makeatother

\begin{document}

\title{On the Hochschild homology of\\ $\boldsymbol{\ell^1}$-rapid decay group algebras}
\author{
Alexander Engel
}
\date{}

\maketitle

\vspace*{-3.0\baselineskip}
\begin{center}
\footnotesize{
\textit{
Fakult\"{a}t f\"{u}r Mathematik\\
Universit\"{a}t Regensburg\\
93040 Regensburg, GERMANY\\
\href{mailto:alexander.engel@mathematik.uni-regensburg.de}{alexander.engel@mathematik.uni-regensburg.de}
}}
\end{center}

\begin{abstract}
We show that for many semi-hyperbolic groups the decomposition into conjugacy classes of the Hochschild homology of the $\ell^1$-rapid decay group algebra is injective.
\end{abstract}

\tableofcontents

\section{Introduction}

Let $G$ be a finitely generated group and denote by $\IC G$ its complex group ring. Choosing any finite, symmetric generating set for $G$, we get a word-length norm on it and can then define the $\ell^1$-rapid decay group algebra $\ell^1_\infty G$, which is the closure of $\IC G$ under a certain family of norms (see Definition~\ref{defn_ell1_RD_algebra} for details).

The \Frechet algebra $\ell^1_\infty G$ plays an important role in the isomorphism conjectures and related questions: its $K$-theory is the target of the Bost assembly map and is also an intermediate step in the Baum--Connes conjecture \cite{lafforgue}, and its Hochschild and cyclic homology are related to the Bass, Burghelea and idempotent conjectures \cite{ji_ogle_ramsey}.

In this paper we are interested in the Hochschild homology of $\ell^1_\infty G$. In the case of the group ring $\IC G$ Burghelea \cite{burghelea} computed the Hochschild homology completely: $\HH_n(\IC G) \cong \bigoplus_{[g] \in \langle G \rangle} H_n(Z_g;\IC)$, where on the right hand side the sum runs over all the conjugacy classes of $G$ and $Z_g$ is the centralizer of $g$. In the case of $\ell^1_\infty G$ we have a map
\begin{equation}
\HHocont_\ast(\ell^1_\infty G) \to \prod_{x \in \langle G \rangle} \HHocont_\ast(\ell^1_\infty G)_x\label{eq_inj_1_intro}
\end{equation}
into the product over the conjugacy classes, but injectivity might be lost now.\footnote{We do not have surjectivity since on the right hand side we give up any control over the norms across the different conjugacy classes.} This was already noticed in \cite{ji_ogle_ramsey}, where a computation of the single factors was carried out. The main result of this paper is to prove injectivity of the map \eqref{eq_inj_1_intro} for as many groups as possible.\footnote{Note that it is currently an open problem to construct a group for which \eqref{eq_inj_1_intro} is not injective.}

\begin{mainthm*}
Let $G$ be a finitely generated group from one of the following classes:
\begin{enumerate}
\item hyperbolic groups,
\item central extensions of hyperbolic groups,
\item Artin groups of extra-large type,
\item right-angled Artin groups, or
\item hyperbolic relative to a finite collection of groups from the previous four classes.
\end{enumerate}
Then the map \eqref{eq_inj_1_intro} is injective.
\end{mainthm*}

\paragraph{Acknowledgements} I thank Benson Farb for helpful answers about his thesis \cite{farb_thesis} and for clarifications about semi-hyperbolic and bicombable groups, and Süleyman Ka\v{g}an Samurka\c{s} for helpful discussions. Furthermore, I thank Crichton Ogle for very valuable discussions and the resulting improvements of this paper (see especially Remark~\ref{rem_ogle}).

Finally, I thank the anonymous referee for his or her comments.

\section{Review of the homology of complex group rings}
\label{sec_review}

Since $\ell^1_\infty G$ is a completion of $\IC G$, we first have to recall the computation of the Hochschild homology $\HH_*(\IC G)$ before we modify these computations to prove our main theorem from the introduction. But before we do that we have to recall the definition of Hochschild homology itself.

\begin{defn}
Let $A$ be an algebra over $\IC$.

The \emph{Hochschild homology} $\HH_\ast(A)$ is defined as the homology of the complex
\[\ldots \longrightarrow C_n(A) \stackrel{b}\longrightarrow C_{n-1}(A) \stackrel{b}\longrightarrow \ldots \stackrel{b}\longrightarrow C_0(A) \longrightarrow 0,\]
where $C_n(A) = A^{\otimes (n+1)}$
and $b$ is the Hochschild boundary operator
\begin{align*}
b & (a_0 \otimes \cdots \otimes a_n)\\
& = \sum_{j=0}^{n-1} (-1)^j a_0 \otimes \cdots \otimes a_j a_{j+1} \otimes \cdots \otimes a_n + (-1)^n a_n a_0 \otimes a_1 \cdots \otimes a_{n-1}.\qedhere
\end{align*}
\end{defn}


The complete computation of the Hochschild and cyclic homology of group rings is due to Burghelea \cite{burghelea}. One can also consult Khalkhali \cite[Example 3.10.3]{khalkhali_basic} or Loday \cite[Chapter 7.4]{loday_cyclic_hom} or Connes \cite[Example 3.2.$\gamma$]{connes_bible}. We will state below only the result for Hochschild homology, since this is the one we need.

We write $\langle G \rangle$ for the conjugacy classes of $G$, and for an element $g \in G$ we write $Z_g < G$ for its centralizer.

\begin{thm}[{Burghelea \cite{burghelea}}]\label{thm2435ter243}
For all $n \in \IN_0$ we have
\begin{equation}
\label{eq_sdfjsdfjk43}
\HH_n(\IC G) \cong \bigoplus_{[g] \in \langle G \rangle} H_n(Z_g;\IC).
\end{equation}
\end{thm}

\subsection{Details of the computation \texorpdfstring{\RNum{1}}{I}: reduction to the centralizers}\label{sechgertet}

We denote by $C_\bullet(\IC G)$ the Hochschild complex of the group ring $\IC G$. For a conjugacy class $x \in \langle G \rangle$ we denote by $C_\bullet(\IC G)_x$ the $\IC$-linear span of the set $\{(g_0, \ldots, g_n)\colon g_0 \cdots g_n \in x\}$.\footnote{We are using here the isomorphism $\IC G \otimes \cdots \otimes \IC G \cong \IC (G \times \cdots \times G)$ to make sense of this.\label{footnote_iso_tensor_group_algebra}}
\label{page_footnote} Then $C_\bullet(\IC G)_x$ is a subcomplex of $C_\bullet(\IC G)$, and we have a splitting
\begin{equation}
\label{eqr423erwe3}
C_\bullet(\IC G) \cong \bigoplus_{x \in \langle G \rangle} C_\bullet(\IC G)_x
\end{equation}
of the Hochschild complex which is responsible for the corresponding splitting in \eqref{eq_sdfjsdfjk43}.

The proof of the following lemma is based on Ji \cite{ji_module, ji_nilpotency} and Nistor \cite{nistor}, and the basic idea for their arguments can be traced back to Burghelea \cite{burghelea}.

\begin{lem}\label{lemrtetrwre}
For all $h \in x$ the inclusion $C_\bullet(\IC Z_h)_{[h]} \to C_\bullet(\IC G)_x$ induces an isomorphism on Hochschild homology groups.
\end{lem}

\begin{proof}
We define a map $\pi_h\colon C_\bullet(\IC G)_x \to C_\bullet(\IC Z_h)_{[h]}$ which is an inverse on homology.

Let us define a map (which is in general not a homomorphism) $p_h\colon G \to Z_h$ by picking for each coset $y \in Z_h \backslash G$ a representative $s(y)$, i.e., $Z_h \cdot s(y) = y$, and then mapping $g$ to $g s(y)^{-1}$ if $g \in y$. It has the property $p_h(ag) = a p_h(g)$ for all $a \in Z_h$.

Let $(g_0, \ldots, g_n) \in C_\bullet(\IC G)_x$ be given. Then there is $r \in G$ such that $g_0 \cdots g_n = r^{-1} h r$, because $h \in x$. Then we set on generators
\begin{align}
\pi_h( & g_0, \ldots, g_n)\label{eq442442344}\\
& := (p_h(r g_0 \cdots g_n)^{-1} h p_h(r g_0), p_h(r g_0)^{-1} p_h(r g_0 g_1), \ldots, p_h(r g_0 \cdots g_{n-1})^{-1} p_h(r g_0 \cdots g_n)).\notag
\end{align}

One quickly checks that $\pi_h$ maps indeed into $C_\bullet(\IC Z_h)_{[h]}$. It is also well-defined in the sense that if we have two representations $g_0 \cdots g_n = r^{-1} h r = l^{-1} h l$, then $\pi_h( g_0, \ldots, g_n)$ is independent of the choice of $r$ or $l$ for the formula. This follows from the fact that in this situation we have $r = a l$ for some element $a \in Z_h$ and we have already noted above that $p_h(ag) = a p_h(g)$ for all $a \in Z_h$. But the map $\pi_h$ does in general depend on the choice of $p_h$.

A computation\footnote{When doing this keep in mind that at some point one has to apply \eqref{eq442442344} to $(g_n g_0, g_1, \ldots, g_{n-1})$ and we have now $g_n g_0 g_1 \cdots g_{n-1} = g_n r^{-1} h r g_n^{-1}$, i.e., one has to apply in this case \eqref{eq442442344} with a different choice of $r$, namely $r g_n^{-1}$.} shows that $\pi_h$ is a chain map, i.e., commutes with boundary operators, and therefore induces a map on homology groups.

We denote the inclusion $C_\bullet(\IC Z_h)_{[h]} \to C_\bullet(\IC G)_x$ by $\iota_h$. The composition $\pi_h \circ \iota_h$ is the identity map and the composition $\iota_h \circ \pi_h$ is chain homotopic to the identity (the latter will be shown in Section~\ref{sec_chain_homotopies}), which finishes this proof. 
\end{proof}

\subsection{Details of the computation \texorpdfstring{\RNum{2}}{II}: homology of the centralizers}
\label{sec_hom_centralizers}

To compute the group homology $H_\ast(G;\IC)$ we can either use the chain complex given by
\begin{align*}
C_n^\prime(G) & := G^n \text{ with the boundary operator}\\
\partial(g_1, \ldots, g_n) & := (g_2, \ldots, g_n) + \sum_{k=1}^{n-1} (-1)^k (g_1, \ldots, g_k g_{k+1}, \ldots g_n) + (-1)^n (g_1, \ldots, g_{n-1})
\end{align*}
or we use the chain complex given by (here $(-)_G$ denotes the $G$-coinvariants with respect to the diagonal action)
\begin{align}
C_n(G) & := (G^{n+1})_G \text{ with the boundary operator}\label{eq_bar_complex}\\
\partial(1, g_1, \ldots, g_n) & := (g_1, \ldots, g_n) + \sum_{k=1}^n (-1)^k (1, g_1, \ldots, \widehat{g_k}, \ldots, g_n).\notag
\end{align}
An isomorphism between these two chain complexes is given by the chain map
\begin{equation*}
\psi\colon C_n^\prime(G) \to C_n(G), \quad \psi(g_1, \ldots, g_n) := (1, g_1, g_1 g_2, \ldots, g_1 g_2 \cdots g_n)
\end{equation*}
and its inverse
\begin{equation*}
\psi^{-1}\colon C_n(G) \to C_n^\prime(G), \quad \psi^{-1}(1, g_1, \ldots, g_n) = (g_1, g_1^{-1} g_2, g_2^{-1} g_3, \ldots, g_{n-1}^{-1} g_n).
\end{equation*}

Recall the notation from Section~\ref{sechgertet}. We have a chain map $\phi_g\colon C_\bullet^\prime(Z_g) \to C_\bullet(\IC Z_g)_{[g]}$ which is in degree $n$ on generators given by $(g_1, \ldots, g_n) \mapsto ((g_1 \cdots g_n)^{-1}g, g_1, \ldots, g_n)$. Its inverse is given by $(g_0, \ldots, g_n) \mapsto (g_1, \ldots, g_n)$. This shows on the nose that $\phi_g$ induces an isomorphism on homology groups $H_\ast(Z_g;\IC) \cong \HH_\ast(Z_g)_{[g]}$, where we write $\HH_\ast(Z_g)_{[g]}$ for the homology of the complex $C_\bullet(\IC Z_g)_{[g]}$. Combined with Lemma~\ref{lemrtetrwre} and \eqref{eqr423erwe3} we therefore deduce the isomorphism \eqref{eq_sdfjsdfjk43} for Hochschild homology.

For every $x \in \langle G\rangle$ and $h \in x$ the formula for the composition
\begin{equation}
\label{eq_full_comp}
C_\bullet(\IC G)_{x} \xrightarrow{\pi_h} C_\bullet(\IC Z_h)_{[h]} \xrightarrow{\phi_h^{-1}} C_n^\prime(Z_h) \xrightarrow{\psi} C_n(Z_h)
\end{equation}
is as follows: given a generator $(g_0, \ldots, g_n)$ with $g_0 \cdots g_n = r^{-1} g r$, its image under the above composition is the equivariant chain having the value $1$ on the orbit of
\begin{equation}
\label{eq_formula_iso_homology}
( p_h(r g_0), p_h(r g_0 g_1), \ldots, p_h(r g_0 \cdots g_n) ).
\end{equation}

\subsection{Details of the computation \texorpdfstring{\RNum{3}}{III}: the chain homotopies}
\label{sec_chain_homotopies}

In this section we are following the presentation of Ji \cite{ji_nilpotency}, which is itself based on work of Nistor \cite{nistor}.

We consider the chain complex $E_\bullet(G)$ with the chain groups $E_n(G) := G^{n+1}$ and the boundary operator
\[\partial(g_0, \ldots, g_n) := (g_1, \ldots, g_n) + \sum_{k=1}^n (-1)^k (g_0, g_1, \ldots, \widehat{g_k}, \ldots, g_n),\]
i.e., the non-equivariant version of \eqref{eq_bar_complex}.

Let $x$ be a conjugacy class of $G$ and let $h \in x$. Let us denote by $i^E\colon E_\bullet(Z_h) \to E_\bullet(G)$ the inclusion map. To define a chain homotopy inverse $p^E$ to it, recall first the definition of the map $p_h\colon G \to Z_h$ from the beginning of the proof of Lemma~\ref{lemrtetrwre}. Then $p^E$ is the extension to $E_\bullet(G)$ of it, i.e., on generators we have
\[p_h(g_0, \ldots, g_n) := (p_h(g_0), \ldots, p_h(g_n)).\]

To define the chain homotopy from $i_h^E p_h^E$ to the identity on $E_\bullet(G)$, we first define
\[D_0\colon E_0(G) \to E_1(G), \quad D_0(g_0) := (g_0 s(Z_h \cdot g_0)^{-1}, g_0)\]
and check that $(\id - i_h^E p_h^E)(g_0) = (\partial_1 D_0)(g_0)$. Then we define inductively
\[D_n\colon E_n(G) \to E_{n+1}(G), \quad D_n(g_0, \ldots, g_n) := (g_0, (\id - i_h^E p_h^E - D_{n-1} \partial_n)(g_0, \ldots, g_n))\]
and this satisfies $\id - i_h^E p_h^E = D_{n-1} \partial_n + \partial_{n+1} D_n$.

We have a chain map
\begin{align*}
\vartheta_h \colon E_n(G) \to C_n(\IC G)_x, \quad \vartheta_h(g_0, \ldots, g_n) := (g_n^{-1} h g_0, g_0^{-1} g_1, \ldots, g_{n-1}^{-1} g_n).
\end{align*}
The kernel of $\vartheta_h$ is spanned by $\{g\cdot (g_0, \ldots, g_n) - (g_0, \ldots, g_n)\colon \forall i(g_i \in G) \text{ and } g \in Z_h\}$, and $\vartheta_h$ is surjective. It follows that it induces an isomorphism $E_\bullet(G) \otimes_{\IC Z_h} \IC \cong C_\bullet(\IC G)_x$. The maps $i^E_h$, $p^E_h$ and the chain homotopies $D_n$ are all $\IC Z_h$-linear, i.e., they induce maps between the quotient complexes $E_\bullet(G) \otimes_{\IC Z_h} \IC$ and $E_\bullet(Z_h) \otimes_{\IC Z_h} \IC$. The main result of this subsection is now that we have a commuting diagram
\begin{align*}
\xymatrix{
E_\bullet(G) \otimes_{\IC Z_h} \IC \ar[rr]^-{\vartheta_h} \ar@<2ex>[d]^{p^E_h} && C_\bullet(\IC G)_x \ar@<1ex>[d]^{\pi_h}\\
E_\bullet(Z_h) \otimes_{\IC Z_h} \IC \ar[rr]^-{\vartheta_h} \ar[u]^{i^E_h} && C_\bullet(\IC Z_h)_{[h]} \ar@<1ex>[u]^{\iota_h}
}
\end{align*}
This provides us the chain homotopies between the identity and $\iota_h \circ \pi_h$.

\section{Homology of \texorpdfstring{$\boldsymbol{\ell^1}$}{l1}-rapid-decay group algebras}
\label{sec_homology_l1}

Let $G$ be a finitely generated group. We fix a finite, symmetric generating set and get a word-length norm $|-|\colon G \to \IR_{\ge 0}$ on $G$. Note that any two choices of finite, generating sets result in quasi-isometric word-length norms.

\begin{defn}\label{defn_ell1_RD_algebra}
For every $k\in \IN$ we define a norm $\|-\|_{k,1}$ on $\IC G$ by
\[\|f\|_{k,1} := \|(1+|-|)^k \cdot f(-)\|_{\ell^1 G}.\]
We denote by $\ell^1_\infty G$ the closure of $\IC G$ under the family of norms $(\|-\|_{k,1})_{k \in \IN}$.
\end{defn}

For any \Frechet algebra\footnote{A \Frechet algebra is a topological vector space whose topology is Hausdorff and induced by a countable family of semi-norms such that it is complete with respect to this family of semi-norms, and such that multiplication is jointly continuous.} $A$ we define $\HHocont_\ast(A)$ analogously as its algebraic counterpart, but we use the completed projective tensor product to form the chain groups.

For $x \in \langle G\rangle$ and $n\in \IN_0$ we equip $C_n(\IC G)_x \subset C_n(\ell^1_\infty G)$ with the induced subspace norm and denote by $C_n(\ell^1_\infty G)_x$ the completion.\footnote{Analogously to Footnote~\ref{footnote_iso_tensor_group_algebra} on Page~\pageref{page_footnote} we use here the isomorphism $\ell^1(G) \hatotimes \cdots \hatotimes \ell^1(G) \cong \ell^1(G \times \cdots \times G)$.} Instead of the isomorphism \eqref{eqr423erwe3} we now only have an injective chain map
\begin{equation}
\label{eq_splitting_l1_cyclic}
C_\bullet(\ell^1_\infty G) \to \prod_{x \in \langle G \rangle} C_\bullet(\ell^1_\infty G)_x.
\end{equation}
It induces a map
\begin{equation}
\HHocont_\ast(\ell^1_\infty G) \to \prod_{x \in \langle G \rangle} \HHocont_\ast(\ell^1_\infty G)_x\label{eq_inj_1}
\end{equation}
which in theory may not be injective, although no counter-examples to injectivity are currently known. Here we write $\HHocont_\ast(\ell^1_\infty G)_x$ for $\HHocont_\ast(C_\bullet(\ell^1_\infty G)_x)$ to shorten notation.

In Section~\ref{sec_l1_localized} we will compute the single factors in the above decomposition,\footnote{Ji--Ogle--Ramsey \cite{ji_ogle_ramsey} already carried out these computations (in greater generality). We write them down again since our later results depend on how these computations concretely look like.} and in Section~\ref{sec_semihyp_product} we will show that for a certain class of groups \eqref{eq_inj_1} is actually injective.

\subsection{Rapid decay group homology}

We will compute the single factors of \eqref{eq_inj_1} by comparing them with a rapid decay version of group homology:

\begin{defn}
On the chain group $C_n(G)$ we define for each $k\in \IN$ a norm $\|-\|_{k,1}$ by
\[\|c\|_{k,1} := \sum_{(g_1, \ldots, g_n) \in G^n} |c(1, g_1, \ldots, g_n)|\cdot \diam(1, g_1, \ldots, g_n)^k.\]
We equip $C_n(G)$ with the family of norms $(\|-\|_{k,1} + \|d-\|_{k,1})_{k\in \IN}$ and denote its completion by $C^{\RD}_n(G)$. The resulting homology is denoted by $H^{\RD}_\ast(G)$.
\end{defn}

To get a hold on the rapid decay group homology we have to impose a polynomial control on the higher-order Dehn functions of the group $G$. Let us define that now.\footnote{Note that there several different competing definitions for higher-order Dehn functions. We have chosen the version of Ji--Ramsey \cite{ji_ramsey}. Another version is due to Riley \cite{riley} and yet another due to Alonso--Pride--Wang \cite{alonso_wang_pride}.}

\begin{defn}[Higher-order Dehn functions, Ji--Ramsey \cite{ji_ramsey}]
Let $X$ be a simplicial complex. For a simplicial $N$-chain $b$ which is a boundary we denote by $l_f(b)$ the least $\ell^1$-norm of an $(N+1)$-chain $a$ with $\partial a = b$. We denote the $\ell^1$-norm of $b$ by $|b|$. The $N$-th Dehn function $d^N(-)\colon \IN \to \IN \cup \{\infty\}$ of $X$ is defined as
\[d^N(k) := \sup_{|b| \le k} l_f(b),\]
where the supremum runs over all $N$-boundaries $b$ of $X$ with $|b| \le k$.

For a group $G$ we choose a simplicial model for $BG$. The higher-order Dehn functions of $G$ are then defined as the higher-order Dehn functions of $EG$. If $G$ is of type $F_{N+1}$, then all the higher-order Dehn functions $d^n(-)$ up to $n \le N$ have finite values and the growth type (e.g., being asymptotically a polynomial of a certain degree) does not depend on the chosen model for $BG$ with finite $(N+1)$-skeleton \cite[Section~2]{ji_ramsey}.
\end{defn}

\begin{prop}\label{prop_Dehn}
Let $G$ be of type $F_\infty$ and let it have polynomially bounded higher-order Dehn functions.
\begin{enumerate}
\item\label{item_iso_homology_RD} The inclusion $C_\bullet(G) \to C_\bullet^\RD(G)$ induces an isomorphism $H_\ast(G) \cong H_\ast^\RD(G)$.
\item\label{item_higher_dehn_estimate} For every $n\in \IN$ and every $k\in \IN$ exists a constant $C_k > 0$ and $p_k \in \IN$ such that if a chain $c \in C_n^\RD(G)$ is a boundary, then there is a chain $b \in C_{n+1}^\RD(G)$ with $d(b) = c$ and $\|b\|_{k,1} \le C_k \cdot \|c\|_{k+p_k,1}$.

Here $p_k$ is the degree of the $n$-th Dehn function of $G$ and the constant $C_k$ depends on its coefficients.
\end{enumerate}
\end{prop}

\begin{proof}
Proofs of Point \ref{item_iso_homology_RD} are given by Ogle \cite{ogle}, Meyer \cite{meyer} and Ji--Ramsey \cite{ji_ramsey}. A geometric proof of it was provided by the author in \cite[Section~4]{engel_BSNC} (though in this reference the version of Riley is used for the higher-order Dehn functions, the proof also works with the version of Ji--Ramsey).

The proof of Point~\ref{item_higher_dehn_estimate} is contained in the first part of the proof of Theorem 2.6($3 \Rightarrow 1$) of Ji--Ramsey \cite{ji_ramsey}.
\end{proof}

\subsection{Computation of the homology localized at a conjugacy class}
\label{sec_l1_localized}

The results in this section are already known \cite{ji_ogle_ramsey}. But we need their proofs in the next section, hence we have to write everything down again.

Let $G$ be a finitely generated group, $x \in \langle G\rangle$ a conjugacy class, and $h \in x$. Recall that we denote by $Z_h \subset G$ the centralizer of $h$ in $G$. We fix a word-length norm on $G$.

In the proof of Lemma~\ref{lemrtetrwre} we defined a map $p_h\colon G \to Z_h$. It was based on picking for each coset $y \in Z_h\backslash G$ a representative $s(y)$ and then mapping $g$ to $g s(y)^{-1}$ if $g \in y$. We pick now the representative $s(y) \in y$ such that it minimizes the length in its class, i.e., $|s(y)| \le |g|$ for all $g \in y$. This does not necessarily uniquely determine the element $s(y)$, but this is of no problem to us. If we equip $Z_h \subset G$ with the induced subspace norm,\footnote{If $Z_h$ is quasi-isometrically embedded in $G$, then we can also use a word norm on $Z_h$. Then $p_h$ is not necessarily $2$-Lipschitz anymore.} then the map $p_h$ is $2$-Lipschitz, i.e., $|p_h(g)| \le 2|g|$ for all $g \in G$.

\begin{defn}[{\cite[Page~99]{ji_ogle_ramsey}}]\label{defn_conjugacy_bound}
We say that $G$ has a polynomially solvable conjugacy bound at $g \in G$, if there exists a polynomial $P$ such that we have the following: for each $h \in [g]$ exists an $r \in G$ with $h = r^{-1} g r$ and $|r| \le P(|h|)$.
\end{defn}

We let $h \in G$ and equip $Z_h \subset G$ with the induced subspace norm. Then the inclusion of complexes $\iota_h\colon C_\bullet(\ell^1_\infty Z_h)_{[h]} \to C_\bullet(\ell^1_\infty G)_{[h]}$ is continuous. If $G$ has a polynomially solvable conjugacy bound at $h$ and if we pick the elements $s(y)$ as above, then the map $\pi_h$ from \eqref{eq442442344} is a continuous map\footnote{We can write $p_h(r g_0 \cdots g_n)^{-1} h = p_h(h^{-1} r g_0 \cdots g_n)^{-1} = p_h(r)$ to eliminate the appearance of $h$ in the formula for $\pi_h$.} and therefore extends continuously to a map
\begin{equation}
\label{eq_map_pi_h}
\pi_h \colon C_\bullet(\ell^1_\infty G)_{[h]} \to C_\bullet(\ell^1_\infty Z_h)_{[h]}.
\end{equation}
Since the chain homotopies from $\iota_h \circ \pi_h$ to the identity are also continuous,\footnote{One has to redo Section~\ref{sec_chain_homotopies} in the setting of $\ell^1$-rapid decay algebras here. Especially, one has to convince oneself that under the assumptions here we have an isomorphism $E^\RD_\bullet(G) \barotimes_{\ell^1_\infty Z_h} \IC \cong C_\bullet(\ell^1_\infty Z_h)_{[h]}$.} we conclude that $\pi_h$ induces isomorphisms on homology groups.

Let us consider the composition $\psi \circ \phi_h^{-1}$ of the last two maps from \eqref{eq_full_comp} and the inverse of this composition. If we equip $Z_h$ with a word-length norm, then this composition and its inverse are continuous and hence we have an isomorphism
\begin{equation*}
\HHocont_\ast(\ell^1_\infty Z_h)_{[h]} \cong H_\ast^\RD(Z_h).
\end{equation*}

Putting it all together, we have proved the following:
\begin{prop}[{\cite[Corollary~1.4.6]{ji_ogle_ramsey}}]
\label{prop_computation_l1_centralizer}
Let $G$ be a finitely generated group. For an $h \in G$ assume that the centralizer $Z_h$ is quasi-isometrically embedded in $G$, and let $G$ have a polynomially solvable conjugacy bound at $h$. Then
\[\HHocont_\ast(\ell^1_\infty G)_{[h]} \cong H_\ast^\RD(Z_h).\]
\end{prop}
Note that the result of the above proposition was already obtained by Ji--Ogle--Ramsey. They do not have the assumption of $Z_h$ being quasi-isometrically embedded in $G$ because they use the induced metric on $Z_h$ (which might not be quasi-isometric to a word-length metric on it), whereas we equip $Z_h$ with a word-length metric to make the homology groups $H_\ast^\RD(Z_h)$ independent of $G$. Note also that Ji--Ogle--Ramsey were able to remove the assumption on the conjugacy bound in \cite{ji_ogle_ramsey_no_conjugacy_bound}.

Since for the identity $e \in G$ the centralizer is the whole group and since the conjugacy class at $e$ is trivial, we immediately get from the above proposition the following:

\begin{cor}\label{cor_at_e_HH_l1}
Let $G$ be a finitely generated group. Then
\[\HHocont_\ast(\ell^1_\infty G)_{[e]} \cong H_\ast^\RD(G).\]
\end{cor}

\begin{cor}\label{cor_inj_l1_compl_localized}
Let $G$ be of type $F_\infty$ and let it have polynomially bounded higher-order Dehn functions.

Then the map $\HH_\ast(\IC G)_{[e]} \to \HHocont_\ast(\ell^1_\infty G)_{[e]}$ is an isomorphism.
\end{cor}

\begin{proof}
From Section~\ref{sec_review} we know that $\HH_\ast(\IC G)_{[e]} \cong H_\ast(G)$ and from the Corollary~\ref{cor_at_e_HH_l1} that $\HHocont_\ast(\ell^1_\infty G)_{[e]} \cong H_\ast^\RD(G)$. The claimed result follows with Proposition~\ref{prop_Dehn}.\ref{item_iso_homology_RD}.
\end{proof}

\subsection{Injectivity of the product decomposition}
\label{sec_semihyp_product}

The discussion in the previous two sections can be used to prove the following results:

\begin{lem}\label{lem_inj_l1_completion}
Let $G$ be a countable group and pick for each conjugacy class $x \in \langle G\rangle$ an element $h_x \in x$ such that the following holds:
\begin{enumerate}
\item\label{item_one_lem_inj_l1_compl} $G$ has a polynomially solvable conjugacy bound at each $h_x$.\footnote{see Definition~\ref{defn_conjugacy_bound}}
\item The centralizers $Z_{h_x} \subset G$ are all of type $F_\infty$ and satisfy the following two conditions:
\begin{enumerate}
\item Every centralizer admits a word-length norm such that the inclusion $Z_{h_x} \subset G$ is a quasi-isometric embedding.
\item Each centralizer has polynomially bounded higher-order Dehn functions.
\end{enumerate}
\end{enumerate}
Then the map $\HH_\ast(\IC G) \to \HHocont_\ast(\ell^1_\infty G)$ is injective.
\end{lem}

\begin{proof}
We have the following diagram, where the left vertical isomorphism is due to the Theorem~\ref{thm2435ter243} and the right vertical map exists due to our assumptions:
\[\xymatrix{
\HH_\ast(\IC G) \ar[r] \ar[d]_{\cong} & \HHocont_\ast(\ell^1_\infty G) \ar[d]\\
\bigoplus_{[g] \in \langle G \rangle} H_\ast(Z_g;\IC) \ar[r] & \prod_{[g] \in \langle G \rangle} H_\ast^\RD(Z_g;\IC)
}\]
By Proposition~\ref{prop_Dehn}.\ref{item_iso_homology_RD} the lower horizontal map is injective, hence the lemma follows.
\end{proof}

\begin{lem}\label{lem_inj_prod}
Let $G$ be a countable group. For every conjugacy class $x \in \langle G\rangle$ we pick an element $h_x \in x$ minimizing the word-length norm in its conjugacy class. We assume that the following holds:
\begin{enumerate}
\item\label{assumption_pol_conjugacy_bond} There is a polynomial $P(-,-)$ in two variables and $G$ has a polynomially solvable conjugacy bound at $h_x$ with governing polynomial $P(-,|h_x|)$.
\item\label{assumption_centralizers} The centralizers $Z_{h_x} \subset G$ are all of type $F_\infty$ and satisfy the following two conditions:
\begin{enumerate}
\item\label{assumption_pol_qi} There exists a polynomial $P^\prime(-)$ and every centralizer $Z_h$ admits a word-length norm such that the inclusion $Z_{h_x} \subset G$ is a quasi-isometric embedding with constants bounded by $P^\prime(|h_x|)$.
\item\label{assumption_pol_Dehn} For each $n\in \IN$ there is a polynomial $P_n(-,-)$ in two variables and each $Z_{h_x}$ has its $n$-th higher-order Dehn functions bounded from above by $P_n(-,|h_x|)$.
\end{enumerate}
\end{enumerate}
Then the map \eqref{eq_inj_1}, i.e., $\HHocont_\ast(\ell^1_\infty G) \to \prod_{x \in \langle G \rangle} \HHocont_\ast(\ell^1_\infty G)_x$ is injective.
\end{lem}

\begin{proof}
Let $[c] \in \HHocont_n(\ell^1_\infty G)$. We first apply the chain map \eqref{eq_splitting_l1_cyclic} to map $c$ to a cycle in the space $\prod_{x \in \langle G \rangle} C_n(\ell^1_\infty G)_x$. Note that each $C_n(\ell^1_\infty G)_x$ has the induced subspace norms from $C_n(\ell^1_\infty G)$. We fix one norm on $C_n(\ell^1_\infty G)$ for the rest of this proof, and consider the corresponding induced norms on each $C_n(\ell^1_\infty G)_x$. The norm of every factor of the image of $c$ under \eqref{eq_splitting_l1_cyclic} is bounded from above by the norm of $c$.

For each conjugacy class $x \in \langle G\rangle$ we let $h_x \in x$ be as in the assumptions of this lemma, i.e., $h_x$ minimizes the word-length norm in its conjugacy class $x$. To each of the factors $C_n(\ell^1_\infty G)_x$ we now apply the map \eqref{eq_map_pi_h}, i.e., we apply $\pi_{h_x} \colon C_n(\ell^1_\infty G)_x \to C_n(\ell^1_\infty Z_{h_x})_{[h_x]}$. Because of the Assumptions~\ref{assumption_pol_conjugacy_bond} and \ref{assumption_pol_qi} of this lemma we can conclude now that the norms of the resulting chains are bounded from above by a factor times the norm of $c$ and such that the factors for the conjugacy classes grow at most polynomially in $|h_x|$.

We use now the composition of the last two maps from \eqref{eq_full_comp} on each factor and arrive in the space $\prod_{x \in \langle G \rangle} C_n^\RD(Z_{h_x})$. Again we can conclude that the norm of each factor of the image of $c$ in $\prod_{x \in \langle G \rangle} C_n^\RD(Z_{h_x})$ is bounded from above by a constant times the norm of $c$ and the constant grows at most polynomially in $|h_x|$.

Now assume that $[c] \in \HHocont_n(\ell^1_\infty G)$ is mapped to zero under the map \eqref{eq_inj_1}. Then we know that the image of $c$ in $\prod_{x \in \langle G \rangle} C_n^\RD(Z_{h_x})$ is a boundary. By Point~\ref{item_higher_dehn_estimate} of Proposition~\ref{prop_Dehn} in combination with Assumption~\ref{assumption_pol_Dehn} of this lemma this implies that there exist chains $b_x \in C_{n+1}^\RD(G)$ with $d(b_x) = c_x$, where $c_x \in C_n^\RD(Z_{h_x})$ is the $x$-component of the image of $c$, and we have $\|b_x\|_{k,1} \le C_k(x) \cdot \|c_x\|_{k+p_k,1}$, where the constants $C_k(x)$ depend polynomially on $|h_x|$. Here $k \in \IN$ corresponds to the at the beginning of this proof chosen norm.

We now use the inverse of the composition of the last two maps from \eqref{eq_full_comp} on each of the factors to map $\prod_{x \in \langle G \rangle} b_x$ into $\prod_{x \in \langle G \rangle} C_{n+1}(\ell^1_\infty Z_{h_x})_{[h_x]}$. Due to the above estimates on the norms of the chains $b_x$ we conclude that the image of $\prod_{x \in \langle G \rangle} b_x$ in $\prod_{x \in \langle G \rangle} C_{n+1}(\ell^1_\infty Z_{h_x})_{[h_x]}$ has finite norms that grow at most polynomially in $|h_x|$. Now we apply $\prod_{x \in \langle G \rangle} \iota_{h_x}$, where the maps $\iota_{h_x} \colon C_{n+1}(\ell^1_\infty Z_{h_x})_{[h_x]} \to C_{n+1}(\ell^1_\infty G)_{x}$ are the inclusions, to arrive at a chain in the space $\prod_{x \in \langle G \rangle} C_{n+1}(\ell^1_\infty G)_{x}$ with an analogous estimate on the norms. Therefore it assembles to a well-defined chain $b$ in $C_{n+1}(\ell^1_\infty G)$.

The claim is now that the Hochschild boundary of $b$ is $c$. To see this one applies the chain homotopies between the identities and the compositions $\iota_{h_x} \circ \pi_{h_x}$ from Section~\ref{sec_chain_homotopies}. The important observation to do here is that under the assumptions of this lemma they assemble to a single well-defined chain homotopy between the identity and $\prod_{x \in \langle G \rangle} \iota_{h_x} \circ \pi_{h_x}$ with the needed polynomial estimates in $|h_x|$.
\end{proof}

\begin{rem}\label{rem_ogle}
This remark resulted from discussions with Crichton Ogle.

Assumption~\ref{assumption_pol_conjugacy_bond} in Lemma~\ref{lem_inj_prod} is relatively restrictive (at least, when compared to the Assumption~\ref{assumption_centralizers})---in Lemma~\ref{lem_Examples_Assumption_1} we exhibit some classes of groups satisfying it.

But using the techniques of \cite{ji_ogle_ramsey_no_conjugacy_bound}, especially of Section~3 thereof, we can get around Assumption~\ref{assumption_pol_conjugacy_bond}, i.e., Lemma~\ref{lem_inj_prod} even holds without it. Unfortunately, this does not help us much since we need in Theorem~\ref{thm_injectivity_prod} the Assumption~\ref{assumption_pol_conjugacy_bond} to varify Assumption~\ref{assumption_pol_Dehn}.

Analogously, we can drop Assumption~\ref{item_one_lem_inj_l1_compl} in Lemma~\ref{lem_inj_l1_completion} by using \cite[Section~3]{ji_ogle_ramsey_no_conjugacy_bound}. From the arguments in the proof of Theorem~\ref{thm_injectivity_prod} below we then conclude that the map $\HH_\ast(\IC G) \to \HHocont_\ast(\ell^1_\infty G)$ is injective for every semi-hyperbolic group.

If we prove the analogous result for cyclic homology, i.e., that $\HC_\ast(\IC G) \to \HCocont_\ast(\ell^1_\infty G)$ is an injective map for every semi-hyperbolic group (this seems feasible since the results of  \cite{ji_ogle_ramsey,ji_ogle_ramsey_no_conjugacy_bound} hold equally well for cyclic homology), then we get from the respective Chern characters\footnote{And using also that $\ell^1_\infty G$ is holomorphically closed in $\ell^1 G$.} rational injectivity of the Bost assembly map $R\!K^G_*(\underline{EG}) \to K_*(\ell^1 G)$ for all semi-hyperbolic groups.

The injectivity results arising from the previous two paragraphs are different from our main theorem which needs uniform (with polynomial bounds) estimates across the different conjugacy classes, whereas the proofs of the statements from the previous two paragraphs only need certain estimates for each conjugacy class separately (which can be seen, e.g., from the lower line of the diagram in the proof of Lemma~\ref{lem_inj_l1_completion} which only involves an algebraic direct sum, i.e., only finitely many different conjugacy classes for each element of it, in the lower left corner).

Note also that using the results and techniques from \cite{JOR_final} one can give alternative proofs (even in some greater generality) of some of the results mentioned above.
\end{rem}

\section{Examples of groups satisfying the assumptions}

Semi-hyperbolic groups were introduced by Epstein et al.~\cite[Section~3.6]{epstein_et_al} under the name \emph{bicombable}. The term \emph{semi-hyperbolic} was coined by Alonso--Bridson \cite{semihyp}.

\begin{thm}\label{thm_injectivity_prod}
Let $G$ be a semi-hyperbolic group satisfying Assumption~\ref{assumption_pol_conjugacy_bond} of Lemma~\ref{lem_inj_prod}.

Then $G$ satisfies Assumption~\ref{assumption_centralizers} of that lemma, and hence \eqref{eq_inj_1} is injective for $G$.
\end{thm}

\begin{proof}
By \cite[Proposition~III.$\Gamma$.4.14]{bridson_haefliger} we get that each centralizer $Z_{h_x}$ is quasi-convex in $G$. Going into the proof of the cited result, we see that the quasi-convexity constant depends polynomially on $|h_x|$ (here we are already using that Assumption~\ref{assumption_pol_conjugacy_bond} of Lemma~\ref{lem_inj_prod} is satisfied, because otherwise we would only get an exponential dependents).

By \cite[Proposition~III.$\Gamma$.4.12]{bridson_haefliger} we get that each centralizer $Z_{h_x}$ is quasi-isometrically embedded and semi-hyperbolic itself for a certain choice of finite generating set of it. Going again into the proof we see that the constants of the quasi-isometric embedding depend only on the constants of the semi-hyperbolicity of $G$, i.e., Assumption~\ref{assumption_pol_qi} of Lemma~\ref{lem_inj_prod} is satisfied even with a constant instead of a polynomial. The semi-hyperbolicity constants will depend polynomially on $|h_x|$.

The centralizers $Z_{h_x}$ are of type $F_\infty$ since they are combable; Alonso \cite[Theorem~2]{alonso}. That the centralizers $Z_{h_x}$ will have polynomially bounded higher-order Dehn functions was already noticed by Ji--Ramsey \cite[End of 2nd paragraph on p.~257]{ji_ramsey} since they are polynomially combable (in our case they are even quasi-geodesically combable). Since the hyperbolicity constants of the centralizers depend polynomially on $|h_x|$, the polynomial bounds on the higher-order Dehn functions will be polynomial in $|h_x|$.

Hence we have checked Assumption~\ref{assumption_centralizers} of Lemma~\ref{lem_inj_prod}, and the injectivity statement follows from it.
\end{proof}

\begin{ex}\label{ex_semihyp}
The following groups are known to be semi-hyperbolic:
\begin{enumerate}
\item hyperbolic groups (Alonso--Bridson \cite{semihyp}),
\item central extensions of hyperbolic groups (Neumann--Reeves \cite{neumann_reeves}),\footnote{Note that hyperbolicity is important here. The $3$-dimensional integral Heisenberg group is a central extension of $\IZ \times \IZ$ by~$\IZ$, but it is not quasi-geodesically combable (since it has a cubic Dehn function, see e.g.~\cite[Section~8.1]{epstein_et_al} or \cite[Remark~5.9]{gersten_heisenberg_group}) and hence can not be semi-hyperbolic.}
\item CAT(0) groups (Alonso--Bridson \cite{semihyp}),
\item systolic groups (Januszkiewicz--{\'{S}}wi{\c{a}}tkowski \cite{janus_swia}),
\item Artin groups of finite type (Charney \cite{charney}),
\item Artin groups of almost large type (Huang--Osajda \cite{huang_osajda}),
\item right-angled Artin groups (Charney--Davis \cite{RAAG_CATcube} proved they are $\mathrm{CAT}(0)$ cube groups, which were shown to be bi-automatic by Niblo--Reeves \cite{niblo_reeves}, and hence semi-hyperbolicity follows),
\item groups acting geometrically and in an order preserving way on Euclidean buildings of the type $\tilde A_n$, $\tilde B_n$ or $\tilde C_n$ (Noskov \cite{noskov} for the case of groups acting freely, and {\'{S}}wi{\c{a}}tkowski \cite{swiatkowski} for the general case).\qedhere
\end{enumerate}
\end{ex}

The above groups are known to be bi-automatic (in the case of CAT(0) groups we have to restrict to CAT(0) cube groups) and hence they have exponentially solvable conjugacy bounds. Let us compile now results from the literature about which of these groups even satisfy the much stronger Assumption~\ref{assumption_pol_conjugacy_bond} of Lemma~\ref{lem_inj_prod}.

\begin{lem}\label{lem_Examples_Assumption_1}
Let $G$ be a group from one of the following classes of groups:
\begin{enumerate}
\item hyperbolic groups,
\item central extensions of hyperbolic groups,
\item Artin groups of extra-large type,\footnote{Artin groups of extra-large type are especially of almost large type and hence are bi-automatic. An earlier proof of bi-automaticity of Artin groups of extra-large type was given by Peifer \cite{peifer}.} or
\item\label{list_RAAG} right-angled Artin groups.
\end{enumerate}
Then $G$ satisfies Assumption~\ref{assumption_pol_conjugacy_bond} of Lemma~\ref{lem_inj_prod}.
\end{lem}

\begin{proof}
The case of hyperbolic groups follows immediately from \cite[Lemma~III.$\Gamma$.2.9]{bridson_haefliger}. Note that in this case the polynomial $P(-,-)$ can be chosen to be linear in both variables and its coefficients also depend linearly on $\delta$, where $\delta$ is the hyperbolicity constant of $G$.

The extension of the results from hyperbolic groups to central extensions of them was achieved in \cite{sale_extensions}.

That Artin groups of extra-large type have a solvable conjugacy problem was shown in \cite{appel_schupp}, and it follows from the proofs given in \cite{holt_rees_extra_large} that they even satisfy Assumption~\ref{assumption_pol_conjugacy_bond} of Lemma~\ref{lem_inj_prod} (with a polynomial of degree one in both variables).\footnote{In their paper the algorithm to decide conjugacy runs in cubic time, but the final conjugating element will have a linear bound on its length.}

Right-angled Artin groups (and some of their $\mathrm{CAT}(0)$ subgroups) have linearly solvable conjugacy bounds by \cite{liu_zeger,crisp_godelle_wiest}.
\end{proof}

\begin{lem}
Let $G$ be a group which is hyperbolic relative to a finite collection of groups from the four classes in the above Lemma~\ref{lem_Examples_Assumption_1}.

Then $G$ is semi-hyperbolic and satisfies Assumption~\ref{assumption_pol_conjugacy_bond} of Lemma~\ref{lem_inj_prod}.
\end{lem}

\begin{proof}
Building on the thesis of Farb \cite{farb_thesis}, Rebbechi \cite[Theorem~9.1]{rebbechi} showed that if $G$ is hyperbolic relative to a finite collection of bi-automatic subgroups with prefix-closed normal forms, then $G$ itself is bi-automatic. Now all the groups stated in Lemma~\ref{lem_Examples_Assumption_1} are bi-automatic, and an inspection of the corresponding proofs reveals that all these bi-automatic structures can be chosen to have prefix-closed normal forms. Hence $G$ is bi-automatic and therefore semi-hyperbolic.

Assumption~\ref{assumption_pol_conjugacy_bond} of Lemma~\ref{lem_inj_prod} is proven in \cite[Theorem~2.2.10]{ji_ogle_ramsey}.
\end{proof}

\begin{lem}
Let $G$ be a CAT(0) group. Then it satisfies Assumption~\ref{assumption_pol_conjugacy_bond} of Lemma~\ref{lem_inj_prod} for its conjugacy classes of finite order elements.\footnote{It seems that for CAT(0) groups a polynomial bound on the conjugacy problem for elements of infinite order is currently not known. But in the case of groups acting on special CAT(0) cube complexes there is some recent progress \cite[Theorem~B]{abbott_behrstock}.}
\end{lem}

\begin{proof}
Follows immediately from \cite[Corollary~III.$\Gamma$.1.14]{bridson_haefliger}. Note that as for hyperbolic groups the polynomial can be taken to be linear in both variables.
\end{proof}

\bibliography{./Bibliography_Homology_RD_algebras}
\bibliographystyle{amsalpha}

\end{document}